%% file: breaking.tex
\documentclass[11pt,reqno,tbtags,draft]{amsart}
\usepackage{amssymb}
\usepackage{url}
\usepackage[square,numbers]{natbib}
\bibpunct[, ]{[}{]}{;}{n}{,}{,}

\title[Breaking Bivariate Records]
{Breaking Bivariate Records}

\newcommand\urladdrx[1]{{\urladdr{\def~{{\tiny$\sim$}}#1}}}
\author{James Allen Fill}
\address{Department of Applied Mathematics and Statistics,
The Johns Hopkins University,
3400 N.~Charles Street,
Baltimore, MD 21218-2682 USA}
\email{jimfill@jhu.edu}
\urladdrx{http://www.ams.jhu.edu/~fill/}
\thanks{Research for both authors supported by
the Acheson~J.~Duncan Fund for the Advancement of Research in
Statistics.}

\numberwithin{equation}{section}

\allowdisplaybreaks

\theoremstyle{plain}
\newtheorem{theorem}{Theorem}[section]
\newtheorem{lemma}[theorem]{Lemma}
\newtheorem{proposition}[theorem]{Proposition}

\newtheorem{conj}[theorem]{Conjecture}

\theoremstyle{definition}
\newtheorem{example}[theorem]{Example}
\newtheorem{definition}[theorem]{Definition}

\newtheorem{remark}[theorem]{Remark}

\newtheorem*{acks}{Acknowledgements}

\theoremstyle{remark}

\newenvironment{romenumerate}[1][-10pt]{
\addtolength{\leftmargini}{#1}\begin{enumerate}
 }{\end{enumerate}}

\newcounter{oldenumi}
{\setcounter{oldenumi}{\value{enumi}}
\begin{romenumerate} \setcounter{enumi}{\value{oldenumi}}}
{\end{romenumerate}}

\newcounter{thmenumerate}

\newcounter{xenumerate}   

\newcommand{\refT}[1]{Theorem~\ref{#1}}

\newcommand{\refL}[1]{Lemma~\ref{#1}}
\newcommand{\refR}[1]{Remark~\ref{#1}}
\newcommand{\refS}[1]{Section~\ref{#1}}

\newcommand{\refP}[1]{Proposition~\ref{#1}}

\newcommand{\refF}[1]{Figure~\ref{#1}}

\newcommand{\refConj}[1]{Conjecture~\ref{#1}}
\newcommand{\refTab}[1]{Table~\ref{#1}}

\newcommand\marginal[1]{\marginpar{\raggedright\parindent=0pt\tiny #1}}
\newcommand\REM[1]{{\raggedright\texttt{[#1]}\par\marginal{XXX}}}

\begingroup
  \divide\count255 by 60
  \multiply\count255 by -60
  \advance\count255 by \time
  \ifnum \count255 < 10 \xdef\klockan{\the\count1.0\the\count255}
  \else\xdef\klockan{\the\count1.\the\count255}\fi
\endgroup

\newcommand\noqed{\renewcommand{\qed}{}} 

\def\rompar(#1){\textup(#1\textup)}    

\def\xexp(#1){e^{#1}}

\newcommand\half{\tfrac12}

\newcommand\punkt{.\spacefactor=1000}    
\newcommand\iid{i.i.d\punkt}

\newcommand{\as}{a.s\punkt}

\newcommand{\io}{i.o\punkt}

\newcommand{\tend}{\longrightarrow}

\newcommand\Pto{\overset{\mathrm{P}}{\tend}}
\newcommand\asto{\overset{\mathrm{a.s.}}{\tend}}

\newcounter{CC}
\newcounter{cc}

\newcommand\E{\operatorname{\mathbb E{}}}
\renewcommand\P{\operatorname{\mathbb P{}}}
\renewcommand\L{\operatorname{L}}
\newcommand\Var{\operatorname{Var}}

\newcommand\dd{\,\mathrm{d}}
\newcommand\ddx{\mathrm{d}}

\newcommand\ee{\mathbf e}
\newcommand\xx{\mathbf x}
\newcommand\XX{\mathbf X}

\newcommand\doi{D_{01}}

\newcommand\tp{\tilde p}

\newcommand\tI{\widetilde I}

\newcommand\dx{D^*}

\newcommand{\ignore}[1]{}

\usepackage{tikz}
\usepackage{amssymb}
\usetikzlibrary{arrows,positioning}
\usepackage[english]{babel}
\usepackage{pgfplotstable}
\usepackage{pgfplots}
\usepackage{adjustbox}
\usetikzlibrary{calc}
\usepackage{relsize}
\usetikzlibrary{arrows.meta}

\pgfdeclarelayer{background layer}
\pgfdeclarelayer{foreground layer}
\pgfsetlayers{background layer,main,foreground layer}

\tikzset{>={Latex[width=5mm,length=5mm]}}

\pgfplotsset{compat=1.3}

\begin{document}

\date{January~23, 2019}

\maketitle

\begin{abstract}
We establish a fundamental property of bivariate Pareto records for independent observations uniformly distributed in the unit square.  We prove that the asymptotic conditional distribution of the number of records broken by an observation given that the observation sets a record is Geometric with parameter $1/2$. 
\end{abstract}

\section{Introduction and main result}
\label{S:intro}

This paper proves an interesting phenomenon concerning the breaking of bivariate records first observed empirically by Daniel~Q.\ Naiman, whom we thank for an introduction to the problem considered.  We begin with some relevant definitions, taken (with trivial changes) from \cite{Fillboundary(2018), Fillgenerating(2018)}.  Although our attention in this paper will be focused on dimension $d = 2$ (see 
\cite[Conj.~2.2]{Fillgenerating(2018)} for general~$d$), and the approach we utilize seems to be limited to the bivariate case, we begin by giving definitions that apply for general dimension~$d$.

Let ${\bf 1}(E) = \mbox{$1$ or $0$}$ according as~$E$ is true or false.
We write $\ln$ or $\L$ for natural logarithm, $\lg$ for binary logarithm, and 
$\log$ when the base doesn't matter.
For $d$-dimensional vectors $x = (x_1, \dots, x_d)$ and $y = (y_1, \dots, y_d)$,
write $x \prec y$ 
to mean that $x_j < y_j$ 
for $j = 1, \dots, d$. 
The notation $x \succ y$ means $y \prec x$. 

As do Bai et al.~\cite{Bai(2005)}, we find it more convenient (in particular, expressions encountered in their computations and ours are simpler) to consider (equivalently) record-\emph{small}, rather than record-large, values.  Let $\XX^{(1)}, \XX^{(2)}, \dots$ be \iid\ (independent and identically distributed) copies of a random vector~$\XX$ with independent coordinates, each uniformly distributed over the unit interval.

\begin{definition}
\label{D:record}
(a)~We say that $\XX^{(n)}$ is a \emph{Pareto record} (or simply \emph{record}, or that $\XX^{(n)}$ \emph{sets} a record at time~$n$) if $\XX^{(n)} \not\succ \XX^{(i)}$ for all $1 \leq i < n$.

(b)~If $1 \leq j \leq n$, we say that $\XX^{(j)}$ is a \emph{current record} (or \emph{remaining record}, or \emph{minimum}) at time~$n$ if $\XX^{(j)} \not\succ \XX^{(i)}$ for all $i \in [n]$. 

(c)~If $0 \leq k \leq n$, we say that $\XX^{(n)}$ \emph{breaks} (or \emph{kills}) $k$ records if $X^{(n)}$ sets a record and there exist precisely~$k$ values~$j$ with $1 \leq j < n$ such that $\XX^{(j)}$ is a current record at time $n - 1$ but is not a current record at time~$n$.
\end{definition}

For $n \geq 1$ (or $n \geq 0$, with the obvious conventions) let $R_n$ denote the number of records 
$\XX^{(k)}$ with $1 \leq k \leq n$, and let $r_n$ denote the number of remaining records at time~$n$. 

\ignore{
{\bf Is the definition of $\mbox{RS}_n$ needed?  If not, then delete.}
\begin{definition}
\label{D:RS}
The \emph{record-setting region} at time~$n$ is the (random) closed set of points
\[
\mbox{RS}_n := \{x \in [0, 1)^d: x \not\succ \XX^{(i)}\mbox{\ for all $i \in [n]$}\}.
\]
%
\end{definition}
}
\medskip

Here is the main result of this paper.

\begin{theorem}
\label{T:main}
Suppose that independent bivariate observations, each uniformly distributed in $(0, 1)^2$, arrive at times 
$1, 2, \ldots$.  
Let $K_n = -1$ if the $n^{\rm \scriptsize th}$ observation is not a new record, and otherwise let $K_n$ denote the number of remaining records killed by the $n^{\rm \scriptsize th}$ observation.  Then $K_n$, conditionally given $K_n \geq 0$, converges in distribution to 
$G - 1$,
where $G \sim \mbox{\rm Geometric$(1/2)$}$,
as $n \to \infty$.
\end{theorem} 

Equivalently, the conclusion (with asymptotics throughout referring to $n \to \infty$) is that
\begin{equation}
\label{convergence}
\P(K_n = k\,|\,K_n \geq 0) \to 2^{- (k + 1)}\mbox{\ for each (fixed) integer $k \geq 0$}.
\end{equation}
Here is an outline of the proof.
In \refS{S:denominator} we provide a simple and short proof of the well-known result that
\[
\P(K_n \geq 0) = n^{-1} H_n, \quad n \geq 1,
\]
where $H_n = \sum_{i = 1}^n i^{-1}$ denotes the $n^{\rm \scriptsize th}$ harmonic number.
In \refS{S:numerator} (see \refT{T:numerator}) we show that
\begin{equation}
\label{remainder}
\left| \P(K_n = k) - \Big[ 2^{- (k + 1)} n^{-1} H_n - (k - 1) 2^{- (k + 2)} n^{-1} \Big] \right|
\leq \tfrac{1}{2} n^{-2}
\end{equation}
for all $n \geq 1$ and all $k \geq 0$.
The improvement
\begin{equation}
\label{improvement}
\left| \P(K_n = k\,|\,K_n \geq 0) - \Big[ 2^{- (k + 1)} + \alpha_{n, k} \Big] \right| 
\leq \tfrac{1}{2} n^{-1} H_n^{-1}
\end{equation}
to~\eqref{convergence}
then follows immediately, where $\alpha_{n, k}$ is a first-order correction term with
\[
\alpha_{n, k} := - (k - 1) 2^{- (k + 2)} H_n^{-1}
\]
to the Geometric$(1/2)$ probability mass function (pmf) $2^{- (k + 1)}$.  This improvement shows that approximation of the conditional pmf in \refT{T:main} by the uncorrected Geometric$(1/2)$ pmf has (for large~$n$) vanishingly small relative error not just for fixed~$k$, but for $k \equiv k_n = o(\log n)$.  It also shows that the corrected approximation has small relative error for $k \leq \lg n + \lg \log n - \omega(1)$.  Of course we always have $K_n \leq r_{n - 1}$, and, by \cite[Rmk.~4.3(b)]{Fillboundary(2018)} we have 
$r_n = O(\log n)$ almost surely; the corrected approximation thus gives small relative error for rather large values of~$k$ indeed.

As one might expect, the correction terms sum to~$0$.  We observe that the correction is positive (and of largest magnitude in absolute-error terms) when $k = 0$, vanishes when $k = 1$, and is negative (and of nonincreasing magnitude) when $k \geq 2$.

Formulation of \refT{T:main} was motivated by \cite[Table~1]{Fillgenerating(2018)}, 
reproduced here as
\refTab{Table1}.
\refTab{Table1}
tabulates, for the first 100,000 records generated in a single trial, the number of records that break~$k$ remaining records, for each value of~$k$.  The Geometric$(1/2)$ pattern is striking.  
The precise relationship between \refT{T:main} and the phenomenon observed in 
\refTab{Table1} is discussed in \refS{S:more}, where a main conjecture is stated and a possible plan for completing its proof is described.
\medskip  

Throughout, we denote the $n^{\rm \scriptsize th}$ observation $\XX^{(n)}$ simply by $\XX = (X, Y)$ (note:\ \emph{sub}scripted~$\XX$ will have a different later use) and, for any Borel subset~$S$ of $(0, 1)^2$, the number of the first~$n$ observations falling in~$S$ by $N_n(S)$.

\begin{center}
\begin{table}
\begin{tabular}{crc}
$k$&
$N_k$\ \ \ &
$\tilde{p}_k$ \\ \hline
0&50,334&0.50334\\
1&24,667&0.24667\\
2&12,507&0.12507\\
3&63,35&0.06335\\
4&3,040&0.03040\\
5&1,571&0.01571\\
6&782&0.00782\\
7&364&0.00364\\
8&202&0.00202\\
9&94&0.00094\\
10&48&0.00048\\
11&24&0.00024\\
12&18&0.00018\\
13&8&0.00008\\
14&4&0.00004\\
16&1&0.00001\\
17&0&0.00000\\
18&1&0.00001\\
\ & \ & \ \\
\end{tabular}
\caption{Results of a simulation experiment in which \mbox{$M = $ 100,000} bivariate records are generated, and for each new record the number~$k$ of records it breaks is recorded. 
The number of records that break~$k$ current records is denoted by $N_k$, and $\tilde{p}_{M, k} = N_k / M$ is the proportion of the 100,000 records that break~$k$ records.}
\label{Table1}
\end{table}
\end{center}

\section{The probability that $K_n \geq 0$}
\label{S:denominator}

In this section we compute the probability $\P(K_n \geq 0)$ (that the $n^{\rm \scriptsize th}$ observation is a record) exactly and approximate it asymptotically.  This result is already well known, 
but we give a proof for completeness.

\begin{proposition}
For $n \geq 1$ we have
\label{P:denominator}
\[
\P(K_n \geq 0) = n^{-1} H_n.
\]
\end{proposition}

\begin{proof}
We have
\begin{align*}
\P(K_n \geq 0, \XX \in \ddx\xx)
&= \P(N_{n - 1}((0, x) \times (0, y)) = 0, \XX \in \ddx\xx) \\
&= \P(N_{n - 1}((0, x) \times (0, y)) = 0)\,\P(\XX \in \ddx\xx) \\
&= (1 - x y)^{n - 1} \dd x \dd y.
\end{align*}
Integrating, we therefore have
\begin{align*}
\P(K_n \geq 0)
&= \int_{x = 0}^1 \int_{y = 0}^1 (1 - x y)^{n - 1} \dd y \dd x
= n^{-1} \int_{x = 0}^1 x^{-1} [1 - (1 - x)^n] \dd x \\
&= n^{-1} \sum_{j = 0}^{n - 1} \int_{x = 0}^1 (1 - x)^j \dd x = n^{-1} H_n,
\end{align*}
as claimed.
\end{proof}

\section{The probability that $K_n = k$}
\label{S:numerator}
In this section, we compute $\P(K_n= k)$ for $k \geq 0$ exactly and produce the approximation~\eqref{remainder} with its stated error bound.

\subsection{The exact probability}
\label{S:exact}

Over the event $\{K_n = k\}$ (with $k \geq 0$), denote those remaining records at time~$n - 1$ broken by 
$\XX$, in order from southeast to northwest (that is, in decreasing order of first coordinate and increasing order of second coordinate), by $\XX_1 = (X_1, Y_1), \ldots, \XX_k = (X_k, Y_k)$.  Note that if we read \emph{all} the remaining records in order from southeast to northwest, then 
$\XX_1, \ldots, \XX_k$ appear consecutively.

If there are any remaining records at time $n - 1$ with 
second 
coordinate smaller than~$Y$, choose the largest such 
second
coordinate $Y_0$ and denote the corresponding remaining record by $\XX_0 = (X_0, Y_0)$ [and note that then $\XX_0, \ldots,$\ $\XX_k$ appear consecutively]; otherwise, set $\XX_0 = (X_0, Y_0) = \ee_1 := (1, 0)$. 

Similarly, if there are any remaining records at time $n - 1$ with 
first
coordinate smaller than~$X$, choose the largest such 
first
coordinate $X_{k + 1}$ and denote the corresponding remaining record by 
$\XX_{k + 1} = (X_{k + 1}, Y_{k + 1})$ [and note that then $\XX_1, \ldots, \XX_{k + 1}$ appear consecutively]; otherwise, set $\XX_{k + 1} = (X_{k + 1}, Y_{k + 1}) = \ee_2 := (0, 1)$. 

Observe that, (almost surely) over the event $\{K_n= k\}$, we have $X_k > X > X_{k + 1}$ and 
$Y_1 > Y > Y_0$.  In results that follow we will only need to treat three cases: (i)~$\XX_0 \neq \ee_1$ and 
$\XX_{k + 1} \neq \ee_2$; (ii)~$\XX_0 = \ee_1$ and $\XX_{k + 1} \neq \ee_2$;
and (iii)~$\XX_0 = \ee_1$ and $\XX_{k + 1} = \ee_2$.  The fourth case~$\XX_0 \neq \ee_1$ and $\XX_{k + 1} = \ee_2$ can be handled by symmetry with respect to the second case.

Our first result of this section specifies the exact joint distribution of 
$\XX, \XX_0, \dots \XX_{k + 1}$.
We write $n^{\underline{k}}$ for the falling factorial power
\[
n (n - 1) \cdots (n - k + 1) = k! \mbox{${n \choose k}$},
\]
and we introduce the abbreviations
\[
\mbox{$\sum_j^k := \sum_{i = j}^k (x_{i - 1} - x_i) y_i, \qquad \sum^k := \sum_1^k$}
\]
for sums that will appear frequently in the sequel.

\begin{proposition}
\label{P:exact}
\ \\

\vspace{-.1in}
{\rm (i)} For $n \geq k + 3$ and
\[
1 > x_0 > \cdots > x_k > x > x_{k + 1} > 0\mbox{\rm \ \ and\ \ }0 < y_0 < y < y_1 < \cdots < y_{k + 1} < 1
\]
we have
\begin{align*}
\lefteqn{\hspace{-.3in}\P(K_n = k;\,\XX \in \dd\xx;\,
\XX_i \in \ddx\xx_i \mbox{\rm \ for $i = 0, \ldots, k + 1$})} \\
&= (n - 1)^{\underline{k + 2}} 
\left[1 - \left\{ \mbox{$\sum^k$} + x_k y_{k + 1} \right\} \right]^{n - (k + 3)}
\dd\xx \dd\xx_0 \cdots \ddx\xx_{k + 1}.
\end{align*}

{\rm (ii)} For $n \geq k + 2$ and
\[
1 > x_1 \cdots > x_k > x > x_{k + 1} > 0\mbox{\rm \ \ and\ \ }0 < y < y_1 < \cdots < y_{k + 1} < 1
\]
we have 
\begin{align*}
\lefteqn{\hspace{-.3in}\P(K_n = k;\,\XX \in \ddx\xx;\,
\XX_0 = \ee_1;\,\XX_i \in \ddx\xx_i \mbox{\rm \ for $i = 1, \ldots, k + 1$})} \\
&= (n - 1)^{\underline{k + 1}} \left[1 - \left\{ \mbox{$\sum^k$} +x_k y_{k + 1} \right\} \right]^{n - (k + 2)} \dd\xx \dd\xx_1 \cdots \ddx\xx_{k + 1}
\end{align*}
where here $x_0 = 1$.

{\rm (iii)} For $n \geq k + 1$ and
\[
1 > x_1 \cdots > x_k > x > 0\mbox{\rm \ \ and\ \ }0 < y < y_1 < \cdots < y_k < 1
\]
we have
\begin{align*}
\lefteqn{\hspace{-.5in}\P(K_n = k;\,\XX \in \ddx\xx;\,\XX_0 = \ee_1;\,
\XX_i \in \ddx\xx_i \mbox{\rm \ for $i = 1, \ldots, k$};\,\XX_{k + 1} = \ee_2)} \\
&= (n - 1)^{\underline{k}} \left[1 - \left\{ \mbox{$\sum^k$} + x_k \right\} \right]^{n - (k + 1)}
\dd\xx \dd\xx_1 \cdots \ddx\xx_k
\end{align*}
where here $x_0 = 1$.
\end{proposition} 

\begin{proof}
We present only the proof of~(i); the proofs of~(ii) and~(iii) are similar.  
We shall be slightly informal in regard to ``differentials'' in our presentation.  
The key is that the event in question (almost surely) equals the following event:
\begin{equation}
\label{event}
\{N_{n - 1}(\ddx\xx_i) = 1\mbox{\rm \ for $i = 0, \ldots, k + 1$};\ N_{n - 1}(S) = 0;\ \XX \in \ddx\xx\}
\end{equation}
where~$S$ is the following disjoint union of rectangular regions:
\[
S = \cup_{i = 1}^k [(x_i, x_{i - 1}) \times (0, y_i)] \cup [(0, x_k) \times (0, y_{k + 1})].
\]
See \refF{F:figure}.
But the probability of the event~\eqref{event} is
\[
(n - 1)^{\underline{k + 2}} \left[ \prod_{i = 0}^{k + 1} \ddx\xx_i \right] \times [1 - \lambda(S)]^{n - (k + 3)} 
\times \ddx\xx,
\]
which reduces easily to the claimed result. 
\end{proof}

\begin{figure}[htb]
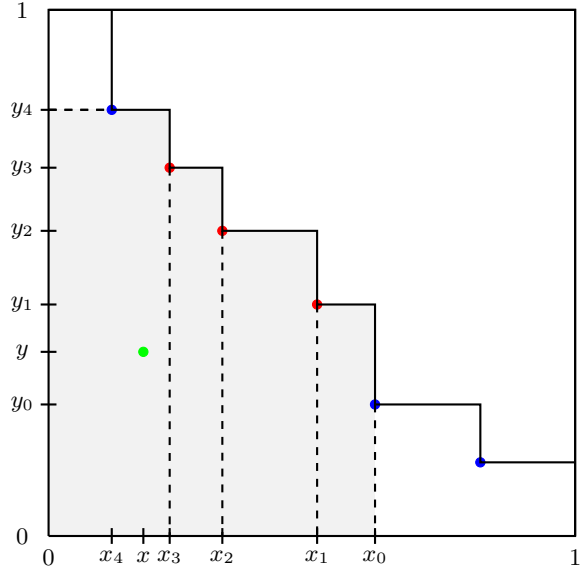

\include{breaking_figure}
\caption{In this example, after $n - 1$ observations, none of which fall in the shaded region~$S$, there are $r_n = 6$ remaining records.  The $n^{\rm \scriptsize th}$ observation, shown in green, breaks the $K_n = k = 3$ remaining records shown in red but not the $r_n - K_n = 3$ remaining records shown in blue.}
\label{F:figure}
\end{figure}

\begin{remark}
\label{R:k=0}
When $k = 0$, \refP{P:exact} is naturally and correctly interpreted as follows:
\smallskip

(i)~For $n \geq 3$ and $1 > x_0 > x > x_1 > 0$ and $0 < y_0 < y < y_1 < 1$ we have
\begin{align*}
\lefteqn{\hspace{-.7in}\P(K_n = 0;\,\XX \in \ddx\xx;\,\XX_0 \in \ddx\xx_0;\, \XX_1 \in \ddx\xx_1)} \\ 
&= (n - 1)^{\underline{2}} (1 - x_0 y_1)^{n - 3} \dd\xx \dd\xx_0 \dd\xx_1.
\end{align*}

{\rm (ii)} For $n \geq 2$ and $1 > x > x_1 > 0$ and $0 < y < y_1 < 1$
we have
\[
\P(K_n = 0;\,\XX \in \ddx\xx;\,\XX_0 = \ee_1;\,\XX_1 \in \ddx\xx_1)
= (n - 1) (1 - y_1)^{n - 2} \dd\xx \dd\xx_1.
\]

{\rm (iii)} For $n \geq 1$ and $1 > x > 0$ and $0 < y < 1$ we have
\[
\P(K_n = 0;\,\XX \in \ddx\xx;\,\XX_0 = \ee_1;\,\XX_1 = \ee_2) = {\bf 1}(n = 1) \dd\xx.
\]
\end{remark}
\medskip

To obtain an exact expression for $\P(K_n = k)$, one need only integrate out the variables $\xx, \xx_i$ in \refP{P:exact} to get
\begin{equation}
\label{abc}
\P(K_n = k) = A_k + 2 B_k + C_k,
\end{equation}
where $A_k$, $B_k$, and $C_k$ (all of which also depend on~$n$) correspond to parts (i), (ii), and~(iii) of the proposition, respectively.  
For small values of~$k$ this can be done explicitly, but for general~$k$ we take an inductive approach.  To get started on the induction, we first treat the case $k = 0$.

\subsection{The case $k = 0$}
\label{S:k=0}
Using \refR{R:k=0}, we obtain the following result.

\begin{proposition}
\label{P:k=0}
We have
\begin{align*}
A_0 &= {\bf 1}(n \geq 3) [\half n^{-1} H_n - \tfrac{3}{4} n^{-1}], \quad B_0 = {\bf 1}(n \geq 2) \half n^{-1}, \quad C_0 = {\bf 1}(n = 1),
\end{align*}
and therefore
\[
\P(K_n = 0) 
= 
\begin{cases}
\half n^{-1} H_n + \tfrac{1}{4} n^{-1} & \mbox{\rm if $n \geq 2$} \\
1 & \mbox{\rm if $n = 1$}.
\end{cases}
\]
\end{proposition}

\begin{proof}
Using \refR{R:k=0}, we perform the computations in increasing order of difficulty.  First, it is clear that $C_0 = 0$ for $n \geq 2$.  Next, for $n \geq 2$ we have
\begin{align*}
B_0
&= \int_{\substack{1 > x > x_1 > 0, \\ 0 < y < y_1 < 1}}\,(n - 1) (1 - y_1)^{n - 2} \dd\xx \dd\xx_1 \\
&= \half (n - 1) \int_{y_1 = 0}^1 y_1 (1 - y_1)^{n - 2}  \dd y_1 = \half n^{-1}. 
\end{align*}

Finally, for $n \geq 3$ we have  
\begin{align*}
A_0
&= \int_{\substack{1 > x_0 > x > x_1 > 0, \\ 0 < y_0 < y < y_1 < 1}}\,(n - 1)^{\underline{2}} (1 - x_0\,y_1)^{n - 3} 
\dd\xx \dd\xx_0 \dd\xx_1 \\
&= \tfrac{1}{4} (n - 1)^{\underline{2}} 
\int_{x_0 = 0}^1 \int_{y_1 = 0}^1 x_0^2\,y_1^2 (1 - x_0\,y_1)^{n - 3} \dd y_1 \dd x_0 \\
&= \tfrac{1}{4} (n - 1)^{\underline{2}} 
\int_{x = 0}^1 x^{-1} \int_{z = 0}^x\!z^2 (1 - z)^{n - 3} \dd z \dd x \\
&= \half n^{-1} \int_{x = 0}^1\!x^{-1} [1 - (1 - x)^n] \dd x \\ 
&{} \qquad \qquad - \half \int_{x = 0}^1\!(1 - x)^{n - 1} \dd x - \tfrac{1}{4} (n - 1) \int_{x = 0}^1\!x (1 - x)^{n - 2} 
\dd x,
\end{align*}
the final equality after two integrations by part.  Using the computation in the proof of \refP{P:denominator} and the above computation of $B_0$, for $n \geq 3$ we therefore find
\begin{align*}
A_0 
&= \half \P(K_n \geq 0) - \half n^{-1} - \half B_0
= \half n^{-1} H_n - \half n^{-1} - \tfrac{1}{4} n^{-1} \\
&= \half n^{-1} H_n - \tfrac{3}{4} n^{-1}.
\end{align*}

Now just use~\eqref{abc} to establish the asserted expression for $\P(K_n = 0)$.
\end{proof}

\subsection{Simplifications}
\label{S:simplifications}

The expressions obtained from \refP{P:exact} for $A_k$, $B_k$, and $C_k$ for $k \geq 1$ are easily simplified by integrating out the four variables $x, x_{k + 1}, y_0, y$ that don't appear in the integrand (when they do appear as variables).  Here is the result.

\begin{lemma}
\label{L:simplifications}
Assume $k \geq 0$.  Let $A_k, B_k, C_k$ be defined as explained at~\eqref{abc}.
\smallskip  

{\rm (i)} For $n \geq k + 3$ we have
\begin{align*}
A_k
&=
\tfrac{1}{4} (n - 1)^{\underline{k + 2}} \\ 
&{} \hspace{-.2in}\times \int_{\substack{1 > x_0 > \cdots > x_k > 0, \\ 0 < y_1 < \cdots < y_{k + 1} < 1}} 
x_k^2 y_1^2
\left[1 - \left\{ \mbox{$\sum^k$} + x_k y_{k + 1} \right\} \right]^{n - (k + 3)}
\dd x_0 \dd\xx_1 \cdots \ddx\xx_k \dd y_{k + 1}.
\end{align*}
\smallskip

{\rm (ii)} For $n \geq k + 2$ we have
\begin{align*}
B_k
&=
\half (n - 1)^{\underline{k + 1}} \\
&{} \times \int_{\substack{1 > x_1 > \cdots > x_k > 0, \\ 0 < y_1 < \cdots < y_{k + 1} < 1}} 
x_k^2 y_1
\left[1 - \left\{ \mbox{$\sum^k$} +x_k y_{k + 1} \right\} \right]^{n - (k + 2)} 
\dd\xx_1 \cdots \ddx\xx_k \dd y_{k + 1},
\end{align*}
where here $x_0 = 1$ and if $k = 0$ then the integral is taken over $0 < y_1 < 1$.
\smallskip 

{\rm (iii)} For $n \geq k + 1$ we have
\begin{align*}
C_k
&= (n - 1)^{\underline{k}}
\int_{\substack{1 > x_1 > \cdots > x_k > 0, \\ 0 < y_1 < \cdots < y_k < 1}}
x_k y_1 
\left[1 - \left\{ \mbox{$\sum^k$} + x_k \right\} \right]^{n - (k + 1)}
\dd\xx_1 \cdots \ddx\xx_k
\end{align*}
where here $x_0 = 1$ and if $k = 0$ then the interpretation is $C_0 = {\bf 1}(n = 1)$.
\end{lemma}

\begin{remark}
\label{R:alternative}
Alternative expressions involving only finite sums are available for $A_k, B_k, C_k$ by recasting the expressions in square brackets in 
Lemma \ref{L:simplifications}
as finite sums of nonnegative terms, expanding the integrand multinomially, and integrating the resulting polynomials explicitly.  When this is done, one finds that $A_k, B_k, C_k$ are all rational, as therefore are 
$\P(K_n = k)$ and $\P(K_n = k\,|\,K_n \geq 0)$.  

Take $C_k$ as an example.  We have
\[
1 - \left\{ \mbox{$\sum^k$} + x_k \right\} = \sum_{i = 1}^k (x_{i - 1} - x_i) (1 - y_i),
\]
and carrying out this procedure yields
\[
C_k = n^{-2} \sum \prod_{i = 1}^k \left( i + \sum_{\ell = k + 1 - i}^k j_{\ell} \right)^{-1},
\]   
where the indicated sum is taken over $k$-tuples $(j_1, \dots, j_k)$ of nonnegative integers summing to 
$n - (k + 1)$ and the natural interpretation for $k = 0$ is $C_0 = {\bf 1}(n = 1)$.  Examples include
\begin{align}
\nonumber
C_1 &= n^{-2} (n - 1)^{-1}, \quad n \geq 2; \\
\nonumber
C_2 &= n^{-2} (n - 1)^{-1} H_{n - 2}, \quad n \geq 3; \\
\label{lastC}
C_{n - 1} &= n^{-2} \prod_{i = 1}^{n - 2} i^{-1} = (n!\,n)^{-1}, \quad n \geq 1.
\end{align}
\end{remark}

Since our aim is to compute $\P(K_n = 0)$ up to additive error $O(n^{-2})$ for large~$n$, the following lemma will suffice to treat the contributions $C_k$.

\begin{lemma}
\label{L:C}
For $n \geq 1$, the probabilities $C_k \geq 0$ satisfy
\[
\sum_{k = 0}^{\infty} C_k = \sum_{k = 0}^{n - 1} C_k = n^{-2}. 
\]
\end{lemma}

\begin{proof}
Recalling that $r_n$ denotes the number of remaining records at time~$n$, it is clear from the description of case~(iii) leading up to \refP{P:exact} that
\[
C_k = \P(r_{n - 1} = k,\,K_n = k) = \P(r_{n - 1} = k, K_n = r_{n - 1}).
\]
Therefore
\[
\sum_{k = 0}^{\infty} C_k = \P(K_n = r_{n - 1}) = \P(\XX \prec \XX^{(i)}\mbox{\ for all $1 \leq i \leq n - 1$})
= n^{-2}.~\qed
\]
\noqed
\end{proof}
 
\subsection{Recurrence relations}
\label{S:recurrence}

In this subsection we establish recurrence relations for $A_k$ and $B_k$ in the variable~$k$, holding~$n$ fixed and treating the probabilities $C_k$ as known.

\begin{lemma}
\label{L:recurrence}
For $k \geq 1$ we have
\begin{enumerate}
\item[(i)] $A_k = \half (A_{k - 1} - B_k)$ if $n \geq k + 3$, 
\item[(ii)] $B_k = \half (B_{k - 1} - C_k)$ if $n \geq k + 2$.
\end{enumerate}
\end{lemma}

\begin{proof}
(i)~Begin with the expression for $A_k$ in \refL{L:simplifications} and integrate out the variable $x_0$.  This gives
\begin{align*}
A_k
&=
\tfrac{1}{4} (n - 1)^{\underline{k + 1}} \\ 
&{} \hspace{-.2in}\times 
\left( 
\int_{\substack{1 > x_1 > \cdots > x_k > 0, \\ 0 < y_1 < \cdots < y_{k + 1} < 1}} 
x_k^2 y_1
\left[1 - \left\{ \mbox{$\sum_2^k$} + x_k y_{k + 1} \right\} \right]^{n - (k + 2)}
\dd\xx_1 \cdots \ddx\xx_k \dd y_{k + 1} 
\right. \\
&{} \hspace{-.08in} - 
\left.
\int_{\substack{1 > x_1 > \cdots > x_k > 0, \\ 0 < y_1 < \cdots < y_{k + 1} < 1}} 
x_k^2 y_1
\left[1 - \left\{ \mbox{$\sum_1^k$} + x_k y_{k + 1} \right\} \right]^{n - (k + 2)}
\dd\xx_1 \cdots \ddx\xx_k \dd y_{k + 1}
\right) \\
&= A_k' - A_k''\mbox{\ (say)},
\end{align*}
with $x_0 = 1$ in the subtracted integral.  For $A_k'$, observe that the variable $y_1$ does not appear within the square brackets in the integrand.  Thus, integrating out $y_1$ and then shifting variable names, we find
\begin{align*}
A_k'
&=
\tfrac{1}{8} (n - 1)^{\underline{k + 1}} \\ 
&{} \hspace{-.2in}\times
\int_{\substack{1 > x_1 > \cdots > x_k > 0, \\ 0 < y_2 < \cdots < y_{k + 1} < 1}} 
x_k^2 y_2^2
\left[1 - \left\{ \mbox{$\sum_2^k$} + x_k y_{k + 1} \right\} \right]^{n - (k + 2)}
\dd x_1 \dd\xx_2 \cdots \ddx\xx_k \dd y_{k + 1} \\
&=
\tfrac{1}{8} (n - 1)^{\underline{k + 1}} 
\int_{\substack{1 > x_0 > \cdots > x_{k - 1} > 0, \\ 0 < y_1 < \cdots < y_k < 1}} 
x_{k - 1}^2 y_1^2 \\
&{} \qquad \qquad \qquad \times 
\left[1 - \left\{ \mbox{$\sum^{k - 1}$} + x_{k - 1} y_k \right\} \right]^{n - (k + 2)}
\dd x_0 \dd\xx_1 \cdots \ddx\xx_{k - 1} \dd y_k \\
&= \half A_{k - 1},
\end{align*}
where the last equality follows from \refL{L:simplifications}.  We see also from \refL{L:simplifications} that 
$A_k'' = \half B_k$.  This completes the proof of part~(i).

(ii)~The proof of part~(ii) is similar.  Begin with the expression for $B_k$ in \refL{L:simplifications} and integrate out the variable $y_{k + 1}$.  This gives (with $x_0 = 1$)
\begin{align*}
B_k
&=
\half (n - 1)^{\underline{k}}
\left(
\int_{\substack{1 > x_1 > \cdots > x_k > 0, \\ 0 < y_1 < \cdots < y_k < 1}} 
x_k y_1
\left[1 - \left\{ \mbox{$\sum^k$} +x_k y_k \right\} \right]^{n - (k + 1)} 
\dd\xx_1 \cdots \ddx\xx_k
\right. \\
&{} \qquad \qquad \quad -
\left.
\int_{\substack{1 > x_1 > \cdots > x_k > 0, \\ 0 < y_1 < \cdots < y_k < 1}} 
x_k y_1
\left[1 - \left\{ \mbox{$\sum^k$} +x_k \right\} \right]^{n - (k + 1)} 
\dd\xx_1 \cdots \ddx\xx_k
\right) \\
&= B_k' - B_k''\mbox{\ (say)}.
\end{align*}
For $B_k'$, observe that the expression within $\{\cdot\}$ equals $\sum^{k - 1} + x_{k - 1} y_k$, which doesn't depend on $x_k$.
Thus, integrating out $x_k$, we find
\begin{align*}
B_k'
&=
\tfrac{1}{4} (n - 1)^{\underline{k}}
\int_{\substack{1 > x_1 > \cdots > x_{k - 1} > 0, \\ 0 < y_1 < \cdots < y_k < 1}} 
x_{k - 1}^2 y_1 \\
&{} \qquad \qquad \qquad \quad \times
\left[1 - \left\{ \mbox{$\sum^{k - 1}$} +x_{k - 1} y_k \right\} \right]^{n - (k + 1)} 
\dd\xx_1 \cdots \ddx\xx_{k - 1} \dd y_k \\
&= \half B_{k - 1},
\end{align*}
where the last equality follows from \refL{L:simplifications}.  We see also from \refL{L:simplifications} that 
$B_k'' = \half C_k$.  This completes the proof of part~(ii).
\end{proof}

The recurrence relations of \refL{L:recurrence} are trivial to solve in terms of the probabilities $C_k$ and the ``initial conditions'' delivered by \refP{P:k=0}.

\begin{lemma}
\label{L:solution}
For $n \geq 1$ and $k \geq 0$ we have
\begin{align}
\label{asolution}
A_k
&= {\bf 1}(n \geq k + 3) \\ 
&{} \qquad \times \left[ 2^{-k} A_0 - k 2^{-(k + 1)} B_0 + \sum_{j = 1}^k (k + 1- j) 2^{- (k + 2 - j)} C_j \right], \nonumber \\
\label{bsolution}
B_k &= {\bf 1}(n \geq k + 2) \left[ 2^{-k} B_0 - \sum_{j = 1}^k 2^{- (k + 1 - j)} C_j \right].
\end{align}
\end{lemma}

\begin{proof}
Clearly we have~\eqref{bsolution} and likewise
\begin{equation}
\label{ab equation}
A_k = 2^{-k} A_0 - \sum_{j = 1}^k 2^{- (k + 1 - j)} B_j.
\end{equation}
Then plugging~\eqref{bsolution} into~\eqref{ab equation} and rearranging yields~\eqref{asolution}. 
\end{proof}

\subsection{Approximation to the probability $\P(K_n = k)$, with error bound.}
\label{S:large n}

\begin{theorem}
\label{T:numerator}
For $n \geq 1$ and every $k \geq 0$ we have
\begin{equation}
\label{remainder}
\left| \P(K_n = k) - \Big[ 2^{- (k + 1)} n^{-1} H_n - (k - 1) 2^{- (k + 2)} n^{-1} \Big] \right|
\leq \tfrac{1}{2} n^{-2}.
\end{equation}
\end{theorem}

\begin{proof}
Recall from~\eqref{abc} that $\P(K_n = k) = A_k + 2 B_k + C_k$; substitute for $A_k$ and $B_k$ using \refL{L:solution}; then substitute for $A_0$ and $B_0$ using \refP{P:k=0}; and finally rearrange.  

For $0 \leq k \leq n - 3$ this gives 
\begin{align*}
\P(K_n = k) 
&= 2^{-k} A_0 - (k - 4) 2^{-(k + 1)} B_0 + \sum_{j = 1}^{k - 1} (k - 3- j) 2^{- (k + 2 - j)} C_j + \tfrac{1}{4} C_k \\
&= 2^{- (k + 1)} n^{-1} H_n - (k - 1) 2^{- (k + 2)} n^{-1} \\ 
&{} \qquad \qquad + \sum_{j = 1}^{k - 1} (k - 3- j) 2^{- (k + 2 - j)} C_j + \tfrac{1}{4} C_k.
\end{align*}
Denote the coefficient of $C_j$ (with $1 \leq j \leq k$) by $c_{k, j}$.  Note that $c_{k, j} \equiv c_{k - j}$ depends only on $k - j \geq 0$, and that $|c_i| \leq 1/4$ (with equality for $c_0 = 1/4$ and $c_1 = -1/4$).  So \refL{L:C} gives the bound on the remainder term (with half as big a constant).

For $k = n - 2$ this gives
\[
\P(K_n = k) = 2^{-k} n^{-1} - \sum_{j = 1}^{k - 1} 2^{- (k - j)} C_j.
\] 
A simple argument omitted here shows that this differs from the approximation in the statement of the theorem by at most $\frac{1}{2} n^{-2}$ for all $n \geq 1$. 

For $k = n - 1$ 
\ignore{
we have $0 \leq \P(K_n = n - 1) \leq \P(r_{n - 1} = n - 1) = 1 / (n - 1)!$, because it is well known that (i)~$r_n$ in dimension~$d$ has the same distribution as $R_n$ in dimension $d - 1$ (the argument for this uses concomitants) and (ii)~$R_n$ in dimension~$1$ is a Poisson-binomial sum, with distribution the convolution of Bernoulli distributions with success probabilities $1 / i$ for $i = 1, \dots, n$.  
}
this together with~\eqref{lastC} gives
\[
\P(K_n = k) = C_{n - 1} = (n!\,n)^{-1}. 
\]
Now another simple and omitted argument shows that this differs from the approximation in the statement of the theorem by at most $\frac{1}{4} n^{-2}$ for all $n \geq 1$.

For $k \geq n$ we have $\P(K_n = k) = 0$, and another simple argument shows that this differs from the asserted approximation by at most $\frac{1}{2} n^{-2}$ provided $n \geq 6$, the worst case being $k = 7$ for $n = 6$ and $k = n$ for $n \geq 7$.  Further, the bound can be checked directly for $n = 1, 2, 3, 4, 5$, the worst~$k$ in each of those cases again being $k = n$.
\end{proof}

\begin{example}
The matrix $C = C_{n, k}$ with $1 \leq n \leq 5$ and $0 \leq k \leq 4$ is
\[
\left[
\begin{array}{ccccc}
1 & {} & {} & {} & {} \vspace{.03in} \\
0 & \tfrac{1}{4} & {} & {} & {} \vspace{.03in} \\
0 & \tfrac{1}{18} & \tfrac{1}{18} & {} & {} \vspace{.03in} \\
0 & \tfrac{1}{48} & \tfrac{1}{32} & \tfrac{1}{96} & {} \vspace{.03in} \\
0 & \tfrac{1}{100} & \tfrac{11}{600} & \tfrac{1}{100} & \tfrac{1}{600}
\end{array}
\right].
\]
Observe that the $n^{\rm \scriptsize th}$ row sums to $n^{-2}$, as noted at \refL{L:C}. 
The matrix with entries $\P(K_n = k)$ for the same values of~$n$ and~$k$ is
\begin{equation}
\label{pnk}
\left[
\begin{array}{ccccc}
1 & {} & {} & {} & {} \vspace{.03in} \\
\tfrac{1}{2} & \tfrac{1}{4} & {} & {} & {} \vspace{.03in} \\
\tfrac{7}{18} & \tfrac{1}{6} & \tfrac{1}{18} & {} & {} \vspace{.03in} \\
\tfrac{31}{96} & \tfrac{13}{96} & \tfrac{5}{96} & \tfrac{1}{96} & {} \vspace{.03in} \\
\tfrac{167}{600} & \tfrac{7}{60} & \tfrac{7}{150} & \tfrac{1}{75} & \tfrac{1}{600}
\end{array}
\right].
\end{equation}
Observe that the $n^{\rm \scriptsize th}$ row sums to $n^{-1} H_n$, as guaranteed by \refP{P:denominator}.
The matrix with entries $\P(K_n = k\,|\,K_n \geq 0)$ is therefore
\[
\left[
\begin{array}{ccccc}
1 & {} & {} & {} & {} \vspace{.03in} \\
\tfrac{2}{3} & \tfrac{1}{3} & {} & {} & {} \vspace{.03in} \\
\tfrac{7}{11} & \tfrac{3}{11} & \tfrac{1}{11} & {} & {} \vspace{.03in} \\
\tfrac{31}{50} & \tfrac{13}{50} & \tfrac{5}{50} & \tfrac{1}{50} & {} \vspace{.03in} \\
\tfrac{167}{274} & \tfrac{35}{137} & \tfrac{14}{137} & \tfrac{4}{137} & \tfrac{1}{274}
\end{array}
\right],
\] 
with every row summing to unity.
\end{example}

\begin{remark}
\label{R:rate}
(a)~Not that the optimal numerical constant appearing on the right in~\eqref{remainder} is important to know, but it would appear from~\eqref{pnk} and other computations that the optimal constant is $1/4$, achieved in four cases: $n = 1, 2$ with $k = n - 1, n$.

(b)~More importantly, we do not know whether the order $n^{-2}$ of the error bound in \refT{T:numerator} is asymptotically optimal.  While the approximation is \emph{perfect} for $k = 0$ if $n \geq 2$, for $k = 1$ it underestimates $\P(K_n = k)$ by $\frac{1}{4} C_1 = \frac{1}{4} n^{-2} (n - 1)^{-1}$ if $n \geq 2$, and for $k = 2$ it underestimates by $\frac{1}{4} (C_2 - C_1) = \frac{1}{4} n^{-2} (n - 1)^{-1} (H_{n - 2} - 1)$ if $n \geq 3$.
Thus the rate of convergence is $O(n^{-2})$ but $\Omega(n^{-3} \log n)$. 

For fixed $k \geq 1$, we conjecture that the correct rate of convergence is $\Theta(n^{-3} (\log n)^{k - 1})$, and more strongly that the error satisfies
\[
\Big[ 2^{- (k + 1)} n^{-1} H_n - (k - 1) 2^{- (k + 2)} n^{-1} \Big] - \P(K_n = k) 
\sim - \tfrac{1}{4} C_k \sim n^{-3} \frac{(\L n)^{k - 1}}{(k - 1)!}  
\] 
as $n \to \infty$.  Since
\[
\sup_{k \geq 1} \frac{(\L n)^{k - 1}}{(k - 1)!} = \Theta\left( \frac{n}{\sqrt{\log n}} \right),
\]
this suggests that perhaps the optimal rate (uniformly in~$k$) for \refT{T:numerator} is the small improvement 
$\Theta(n^{-2} (\log n)^{-1/2})$.
\end{remark}

\section{Conjectures}
\label{S:more}

The upshot of this section is that a variance bound would imply a Glivenko--Cantelli type theorem: \refConj{conj:variance0} would imply \refConj{conj:GC}.

\subsection{The natural conjecture}
\label{S:natural}

While our main \refT{T:main} does begin to explain how the Geometric$(1/2)$ distribution arises in connection with the breaking of bivariate records, it is not the conjecture to which one is led by performing many independent trials of generating a large number~$M$ of records and, for each trial, watching the table such 
as \refTab{Table1} evolve as records are generated one at at a time.  A natural conjecture concerns the fractions of records that break~$k$ remaining records, for various values of~$k$.  Accordingly, let 
\[
\tp_{M, k} := M^{-1} \sum_{m = 1}^M \tI_{m, k}
\]
where
\[
\tI_{m, k} := {\bf 1}(\mbox{$m^{\mathrm{th}}$ record generated breaks precisely~$k$ remaining records}).
\]
A strong conjecture one might form is the following, of Glivenko--Cantelli type:

\begin{conj}
\label{conj:GC}
The fractions~$\tp_{M, k}$ of the first~$M$ records that break precisely~$k$ remaining records satisfy
\[
\sup_{k \geq 0} \left| \tp_{M, k} - 2^{ - (k + 1)} \right| \asto 0 \mbox{\rm \ as $M \to \infty$}.
\]
\end{conj}

In the remaining subsections we show how proving this conjecture can be reduced to an asymptotic variance calculation, and we leave that calculation for future research.

\subsection{Uniformity in~$k$}
\label{S:uniformity}

Of course, \refConj{conj:GC} would have the following corollary, of strong law of large numbers type.

\begin{conj}
\label{conj:SLLN}
For each fixed~$k \geq 0$, the fraction~$\tp_{M, k}$ of the first~$M$ records that breaks precisely~$k$ remaining records satisfies
\[
\tp_{M, k} \asto 2^{ - (k + 1)}\mbox{\rm \ as $M \to \infty$}.
\]
\end{conj}

But it is standard to check that \refConj{conj:SLLN} also implies \refConj{conj:GC}.  For completeness, here is a proof, with all claims holding almost surely.  Let $\epsilon_{M, k} \geq 0$ denote the random variable 
$|\tp_{M, k} - 2^{ - (k + 1)}|$.  Then for any $K \geq 0$ we have
\[
\epsilon_M := \sup_{k \geq 0} \epsilon_{M, k} 
= \max\left\{ \max_{k \leq K} \epsilon_{M, k},\ \sup_{k > K} \epsilon_{M, k} \right\}
= \sup_{k > K} \epsilon_{M, k}
\]     
by \refConj{conj:SLLN}.  But
\[
\sup_{k > K} \epsilon_{M, k} \leq \sum_{k > K} \tp_{M, k} + 2^{- (K + 1)} 
= 1 - \sum_{k \leq K} \tp_{M, k} + 2^{- (K + 1)}.
\]  
Therefore
\begin{align*}
\limsup_{M \to \infty} \epsilon_M 
&\leq 1 - \sum_{k \leq K} \lim_{M \to \infty} \tp_{M, k} + 2^{- (K + 1)} \\
&= 1 - \sum_{k \leq K} 2^{ - (k + 1)} + 2^{- (K + 1)}
= 2^{-K}.
\end{align*}
Letting $K \to \infty$ completes the proof.~\qed

\subsection{Time change}
\label{S:time}

We show next that \refConj{conj:SLLN} would follow from the following ``observations-time'' conjecture.  Let
\begin{equation}
\label{Rnkdef}
R_{n, k} := \sum_{i = 1}^n I_{i, k}
\end{equation}
where
\[
I_{i, k} := {\bf 1}(K_i = k).
\]
Note that
\[
R_n = \sum_{k \geq 0} R_{n, k},
\]
and define
\[
p_{n, k} := \frac{R_{n, k}}{R_n}.
\]

\begin{conj}
\label{conj:obstime}
For each fixed $k \geq 0$ we have
\[
p_{n, k} \asto 2^{- (k + 1)}\mbox{\rm \ as $n \to \infty$}.
\] 
\end{conj}

Here is a proof that \refConj{conj:obstime} implies \refConj{conj:SLLN}.  Working in observations-time, for 
$m \geq 1$, let $T_m$ denote the time at which the $m^{\rm \scriptsize th}$ record is set, so that $R_{T_m} = m$ for all~$m$.  In similar fashion, $R_{T_M, k} = \sum_{m= 1}^M \tI_{m, k}$.  Thus \refConj{conj:SLLN} follows from \refConj{conj:obstime} simply by looking at the sequence $(T_m)$ of $n$-values.~\qed

\subsection{Expectations}
\label{S:expectations}

\refConj{conj:obstime} is certainly plausible, because, as we prove in this subsection, with
\[
\rho_{n, k} := \E R_{n, k}, \quad \rho_n := \E R_n, \quad \phi_{n, k} := \frac{\rho_{n, k}}{\rho_n}
\]
we have
\begin{equation}
\label{phi}
\phi_{n, k} \to 2^{- (k + 1)}\mbox{\rm \ as $n \to \infty$}.
\end{equation}
In the statement of the following lemma, we refer (indirectly) to the second-order harmonic numbers
\[ 
H_n^{(2)} = \frac{\pi^2}{6} - (1 + o(1)) n^{-1}\mbox{\rm\ as $n \to \infty$}, \quad \mbox{where} \quad
H_n^{(r)} := \sum_{i = 1}^n i^{-r}
\]
(aside:\ we shall encounter the fourth-order harmonic numbers in \refS{S:variance})
and (directly) to the second-order \emph{Roman harmonic numbers} (cf.~\cite{Sesma(2017)} and references [16, 22, 23] therein)
\begin{align*}
c_n^{(2)} 
&:= \sum_{i = 1}^n i^{-1} H_i = \half (H_n^2 + H_n^{(2)}) \\ 
&= \half (\L n)^2 + \gamma \L n + \frac{1}{2} \left( \frac{\pi^2}{6} + \gamma^2 \right) + O(n^{-1} \log n).
\end{align*}
The lemma shows that
\[
\hat{\rho}_{n, k} := 2^{- (k + 1)} c_n^{(2)} - (k - 1) 2^{- (k + 2)} H_n
\] 
gives a good approximation to $\rho_{n, k}$.  

\begin{lemma}
\label{L:expectations}
For $n \geq 1$ we have
\begin{equation}
\label{rhon}
\rho_n = c_n^{(2)}
\end{equation}
and, for every $k \geq 0$, also
\begin{equation}
\label{rhonk}
|\hat{\rho}_{n, k} - \rho_{n, k}| \leq \half H_n^{(2)} < 1.
\end{equation}
\end{lemma} 

\begin{proof}
For~\eqref{rhon}, just sum the result of \refP{P:denominator} (with~$n$ replaced by~$i$) over~$i$ from~$1$ to~$n$. 
For~\eqref{rhonk}, apply the same operation to~\eqref{remainder} in \refT{T:numerator}, observing 
$\pi^2 / 12 < 1$.
\end{proof}

\begin{remark}
\label{R:phi}
From \refL{L:expectations} it is an immediate corollary that
\[
\sup_{k \geq 0} \left| \phi_{n, k} - \left[ 2^{- (k + 1)} - (k - 1) 2^{- (k + 2)} \frac{H_n}{c_n^{(2)}} \right] \right| 
< \frac{1}{c_n^{(2)}} \sim (\L n)^{-2};
\]
in particular, \eqref{phi} holds, uniformly in~$k$.
\end{remark}

\subsection{Reduction to a variance calculation}
\label{S:reduction}

In light of \refL{L:expectations}, to establish $p_{n, k} \Pto 2^{- (k + 1)}$ as $n \to \infty$ it would be sufficient to establish concentration of measure for the distributions of the denominator $R_n$ and the numerator 
$R_{n, k}$ of $p_{n, k}$---for example, by means of variance bounds combined with Chebyshev's inequality.  As we will explain in this subsection, we already know about the variance of $R_n$, and if we were to bound the variance of $R_{n, k}$ in suitably similar fashion we could prove not only convergence in probability but also the almost sure convergence of \refConj{conj:SLLN}.

The following results concerning $R_n$ are implied by \cite[Thms.~4.1(b), 4.2(a)]{Fillboundary(2018)} (with the mean, variance, and central limit theorem results there taken from Bai et al.~\cite{Bai(2005), Bai(1998)}) after specializing to our present case of dimension $d = 2$.

\begin{lemma}
\label{L:Rn}
Let~$\Phi$ denote the standard normal distribution function.  The number $R_n$ of records set through time~$n$ satisfies
\begin{align*}
\rho_n &= \E R_n = \tfrac{1}{2} (\L n)^2 + \gamma \L n + \left( \tfrac{\pi^2}{12} + \half \gamma^2 \right) + o(1), \\ 
\sigma_n^2 &:= \Var R_n \sim \left( \tfrac{\pi^2}{6} + \gamma^2 \right) (\L n)^2,
\end{align*}
\[
\sup_x \left| \P\left( \frac{R_n - \rho_n}{\sigma_n} < x \right) - \Phi(x) \right|
= O((\log n)^{-  1 / 2} (\log \log n)^3),
\]
\begin{equation}
\label{LD}
\P\left( |R_n - \rho_n| \geq (\L n)^{\frac{3}{2} + \epsilon}\mbox{\rm \ \io} \right) 
= 0\mbox{\rm \ if $\epsilon > 0$},
\end{equation}
and consequently
\begin{equation}
\label{as}
\frac{R_n}{\rho_n} \asto 1.~\qed
\end{equation}
\end{lemma}

A careful review of the proof of~\eqref{LD} (a first Borel--Cantelli argument applied along a geometrically increasing sequence of times), which immediately implies~\eqref{as}, shows that to establish~\eqref{LD} it is sufficient to know that the samples paths of the process~$R$ are nondecreasing, that
\[
\rho_n = a (\L n)^2 + b (\L n) + O(1)
\]
for some constants $a > 0$ and $b$, that $\sigma^2_n = O((\log n)^2)$, and that
\[
\rho_n - \rho_{n - 1} = \Theta(n^{-1} \log n).
\]
Now observe, for each fixed $k \geq 0$, that the sample paths of the process $R_{\cdot, k}$ are nondecreasing, that
\[
\rho_{n, k} = a_k (\L n)^2 + b_k (\L n) + O(1)
\]
with $a_k = 2^{- (k + 2)} > 0$ and $b_k = - 2^{- (k + 2)} (k - 2 \gamma - 1)$, and that
\[
\rho_{n, k} - \rho_{n - 1, k} = \P(K_n = k) = \Theta(n^{-1} \log n),
\] 
with the last equality holding by \refT{T:numerator}.  Thus the analogues of \eqref{LD}--\eqref{as} 
for $R_{\cdot, k}$ hold if we can establish that
\begin{equation}
\label{variancedef}
\sigma_{n, k}^2 := \Var R_{n, k}
\end{equation}
satisfies $\sigma_{n, k}^2 = O((\log n)^2)$, which (in light of the known corresponding result for~$R$) seems eminently reasonable to conjecture.

\begin{conj}
\label{conj:variance}
For each fixed $k \geq 0$, the variance $\sigma^2_{n, k}$ defined at~\eqref{variancedef} satisfies
\[
\sigma_{n, k}^2 = O((\log n)^2).
\]
\end{conj}

A summary of this subsection is that \refConj{conj:variance} would imply \refConj{conj:obstime} and therefore also \refConj{conj:GC}.

\begin{remark}
\label{R:stronger Rnk conjectures}
(a)~Use of the refinement~\eqref{LD} to~\eqref{as} shows that \refConj{conj:variance} would imply the refinement
\[
p_{n, k} = 2^{- (k + 1)}[1 + O((\log n)^{- (1 / 2) + \epsilon})]\mbox{\rm \ \as}
\]
of \refConj{conj:obstime} for each fixed $k \geq 0$ and any $\epsilon > 0$.

(b)~More than \refConj{conj:variance}, we conjecture that for each fixed $k \geq 0$ we have
\[
\sigma_{n, k}^2 \sim s_k^2 (\L n)^2
\]
for some constants $s_k^2 > 0$ satisfying $s_k^2 \to 0$ as $k \to \infty$ (likely with 
$s_k \equiv 2^{- (k + 1)} s$, letting $s^2 := \frac{\pi^2}{6} + \gamma^2$), and that there is asymptotic normality for $R_{n, k}$.  It seems reasonable to conjecture that, moreover, the random vector $(R_{n, 1}, \dots, R_{n, k})$ enjoys full-dimensional asymptotic $k$-variate normality.

(c)~It may be that the random variables $R_{n, k}$ are positively correlated for fixed~$n$ as~$k$ varies, the idea being that larger values of~$R_n$ (more records) should lead to larger values of $R_{n, k}$ (more records that break~$k$ remaining records) for every~$k$.  If this positive correlation were to be known, then \refConj{conj:variance} would follow immediately, without the need for additional calculations.  Indeed, for large~$n$ and fixed~$k$ we would then have
\[
\sigma_{n, k}^2 \leq \sum_{j = 1}^n \sigma_{n, j}^2 \leq \sigma^2_n \sim s^2 (\L n)^2.
\]
\end{remark}

\subsection{Reduction of the variance calculation}
\label{S:variance}

Corresponding to the breakdown into cases utilized in \refS{S:numerator}, observe that 
$I_{n, k} = {\bf 1}(K_n = k)$ satisfies
\[
I_{n, k} = I^{(0)}_{n, k} + I^{(1)}_{n, k} + I^{(2)}_{n, k} + I^{(1, 2)}_{n, k},
\]
where the four terms here are the respective indicators of the events
\[
\begin{array}{llll}
\{ K_n = k,\ \mbox{$\XX^{(n)}$ does not set a record in either coordinate} \}, \\
\{ K_n = k,\ \mbox{$\XX^{(n)}$ sets a record in the first coordinate but not the second} \}, \\
\{ K_n = k,\ \mbox{$\XX^{(n)}$ sets a record in the second coordinate but not the first} \}, \\
\{ K_n = k,\ \mbox{$\XX^{(n)}$ sets a record in both coordinates} \}. \\
\end{array}
\]
By analogy with~\eqref{Rnkdef}, define respective record counts
$R^{(0)}_{n, k}, R^{(1)}_{n, k}, R^{(2)}_{n, k}, R^{(1, 2)}_{n, k}$, so that
\begin{equation}
\label{Rbreakdown}
R_{n, k} = R^{(0)}_{n, k} + R^{(1)}_{n, k} + R^{(2)}_{n, k} + R^{(1, 2)}_{n, k}.
\end{equation}
It thus seems daunting to calculate $\sigma^2_{n, k}$ to prove~\refConj{conj:variance}.  But in this subsection we argue by means of suitable control of all but the first term in~\eqref{Rbreakdown} that
\[
\sigma_{n, k}^2 = \Var R^{(0)}_{n, k} + O((\log n)^2),
\]
for fixed~$k$, thus reducing proof of \refConj{conj:variance} to proof of the following simpler conjecture.

\begin{conj}
\label{conj:variance0}
For each fixed $k \geq 0$ we have
\[
\Var R^{(0)}_{n, k} = O((\log n)^2).
\]
\end{conj}

Here is a proof that \refConj{conj:variance0} would imply \refConj{conj:variance}.  By the triangle inequality for $L^2$-norm $\| \cdot \|_2$, in obvious notation we have
\begin{equation}
\label{control}
\sigma_{n, k} - \sigma^{(0)}_{n, k}
\leq 
\sigma^{(1)}_{n, k} + \sigma^{(2)}_{n, k} + \sigma^{(1, 2)}_{n, k} 
= 2 \sigma^{(1)}_{n, k} + \sigma^{(1, 2)}_{n, k}
\end{equation}

But, with $R^{(1)}_n$ counting the number of records through time~$n$ in the first coordinate, we have 
\begin{align}
\Var R^{(1)}_{n, k}
&\leq \left\| R^{(1)}_{n, k} \right\|_2^2 
\leq \left\| R^{(1)}_n \right\|_2^2 
= \left[ \E R^{(1)}_n \right]^2 + \Var R^{(1)}_n 
= H_n^2 + [H_n - H^{(2)}_n] \nonumber \\ 
\label{Rnk1}
&= O((\log n)^2);
\end{align}
and, with $R^{(1, 2)}_n$ counting the number of observations through time~$n$ that set a record in both coordinates, we have  
\begin{align}
\Var R^{(1, 2)}_{n, k}
\nonumber
&\leq \left\| R^{(1, 2)}_{n, k} \right\|_2^2 
\leq \left\| R^{(1, 2)}_n \right\|_2^2 
= \left[ \E R^{(1, 2)}_n \right]^2 + \Var R^{(1, 2)}_n \\
\nonumber 
&= (H^{(2)}_n)^2 + [H^{(2)}_n - H^{(4)}_n] \\
\label{R12nk} 
&= O(1) = o((\log n)^2).
\end{align}
Thus, returning to~\eqref{control} and applying the inequality $(a + b)^2 \leq 2 (a^2 + b^2)$, we find
\[
\sigma^2_{n, k} \leq \left[ \sigma^{(0)}_{n, k} + O(\log n) \right]^2 \leq 2 \Var R^{(0)}_{n, k} + O((\log n)^2),
\]
and so \refConj{conj:variance0} would imply \refConj{conj:variance}.~\qed

\begin{remark}
\label{R:smaller}

(a)~Observe that $R^{(1, 2)}_{n, 0} = 1$ for every $n \geq 1$, and so $\Var R^{(1, 2)}_{n, 0} = 0$.
For $k \geq 1$, we claim that~\eqref{R12nk} can be strengthened to $\Var R^{(1, 2)}_{n, k} = \Theta(1)$.  
To establish the lower bound $\Var R^{(1, 2)}_{n, k} = \Omega(1)$ matching the upper bound \eqref{R12nk}, we perform two computations.  The first, valid for $n \geq 2 k + 1$, is that
\[
\P\left( R^{(1, 2)}_{n, k} \geq 2 \right)
\geq \P\left( R^{(1, 2)}_{2 k + 1, k} = 2 \right) 
=  \P\left( R^{(1, 2)}_{k + 1, k} = 1, R^{(1, 2)}_{2 k + 1, k} = 2 \right)  
> 0,
\]
and the other, valid for $n \geq k + 1$, is that
\begin{align*}
\P\left( R^{(1, 2)}_{n, k} = 1 \right)
&\geq \P\left( R^{(1, 2)}_{k + 1, k} = 1,\ R^{(1, 2)}_{n, k} = 1 \right) \\
&\geq \P\left( R^{(1, 2)}_{k + 1, k} = 1 \right) \P\left( R^{(1, 2)}_{n - k} = 1 \right) \\
&= \P\left( R^{(1, 2)}_{k + 1, k} = 1 \right) \mbox{$\prod$}_{i = 2}^{n - k} (1 - i^{-2}) \\
&= \half \P\left( R^{(1, 2)}_{k + 1, k} = 1 \right) [1 + (n - k)]^{-1} \\
&\geq \half \P\left( R^{(1, 2)}_{k + 1, k} = 1 \right) > 0.      
\end{align*}

(b)~We conjecture that~\eqref{Rnk1} can be strengthened to $\Var R^{(1)}_{n, k} = \Theta(\log n)$.  If we knew even the upper bound $\Var R^{(1)}_{n, k} = O(\log n)$, then it would follow from~\eqref{control} and the matching upper bound on $\sigma^{(0)}_{n, k} - \sigma_{n, k}$ that
\[
\sigma_{n, k} = \sigma^{(0)}_{n, k} + O((\log n)^{1/2}).
\]
In that way, if one could prove the conjecture that $\sigma^{(0)}_{n, k} \sim s_k \L n$ for some constant 
$s_k > 0$, then the same lead-order asymptotics would apply to $\sigma_{n, k}$. 
\end{remark}

\begin{acks}
We thank Vince Lyzinski, Daniel~Q.\ Naiman and Fred Torcaso for helpful comments, and Daniel~Q.\ Naiman for producing \refF{F:figure}.
\end{acks}

\bibliography{records.bib}
\bibliographystyle{plain} 
\end{document}

%% file: breaking_figure.tex
\begin{adjustbox}{max totalsize={.9\textwidth}{.7\textheight},center}
\begin{tikzpicture}[scale=7]

\begin{pgfonlayer}{background layer}

%
%

%
%
\fill[fill=gray!10]    (0,0) -- (.23,0.) -- (.23,.81) -- (0.,.81);
\fill[fill=gray!10]    (.23,0) -- (.33,0.) -- (.33,.70) -- (.23,.70);
\fill[fill=gray!10]    (.33,0) -- (.51,0.) -- (.51,.58) -- (.33,.58);
\fill[fill=gray!10]    (.51,0) -- (.62,0.) -- (.62,.44) -- (.51,.44);

%
%
\draw[thick,color=black](0,1.)--(1.,1.)--(1.,0.);

%
%
\filldraw [blue]
(.12,.81) circle (.25pt)
;
\filldraw[red]
(.23,.70) circle (.25pt)
(.33,.58) circle (.25pt)
(.51,.44) circle (.25pt)
;
\filldraw[blue]
(.62,.25) circle (.25pt)
(.82,.14) circle (.25pt)
;

%
%
\filldraw[green]
(.18,.35) circle (.25pt)
;

%
%
\draw[dashed,color=black, thick](0,.81)--(.12,.81);
\draw[dashed,color=black, thick](.23,.70)--(.23,0.);
\draw[dashed,color=black, thick](.33,.58)--(.33,0.);
\draw[dashed,color=black, thick](.51,.44)--(.51,0.);
\draw[dashed,color=black, thick](.62,.25)--(.62,0.);
\draw[dashed,color=black, thick](0,.81)--(.12,.81);

%
%
\draw[thick,color=black]
(.12,1.)--
(.12,.81)--(.23,.81) --
(.23,.70)--(.33,.70) --
(.33,.58)--(.51,.58) --
(.51,.44)-- (.62,.44) --
(.62,.25)-- (.82,.25) --
(.82,.14) -- (1.,.14)
;

%
%
\def\mmx{.015}
\def\mmy{.015}
\draw[color=black,thick] (.12,-\mmx)--(.12,+\mmx);
\draw [color=black,thick] (.18,-\mmx)--(.18,+\mmx);
\draw [color=black,thick] (.23,-\mmx)--(.23,+\mmx);
\draw [color=black,thick] (.33,-\mmx)--(.33,+\mmx);
\draw [color=black,thick] (.51,-\mmx)--(.51,+\mmx);
\draw [color=black,thick] (.62,-\mmx)--(.62,+\mmx);

\draw[color=black,thick] (-\mmy,.81)--(+\mmy,.81);
\draw [color=black,thick] (-\mmy,.70)--(+\mmy,.70);
\draw [color=black,thick] (-\mmy,.58)--(+\mmy,.58);
\draw [color=black,thick] (-\mmy,.44)--(+\mmy,.44);
\draw [color=black,thick] (-\mmy,.25)--(+\mmy,.25);
\draw [color=black,thick] (-\mmy,.35)--(+\mmy,.35);

%
%
\def\dx{0.}
\def\dy{.04}
\draw (0+\dx,-\dy) node[color=black] {\footnotesize $0$};
\draw (.12+\dx,-\dy) node[color=black] {\footnotesize $x_4$};
\draw (.18+\dx,-\dy) node[color=black] {\footnotesize $x$};
\draw (.23+\dx,-\dy) node[color=black] {\footnotesize $x_3$};
\draw (.33+\dx,-\dy) node[color=black] {\footnotesize $x_2$};
\draw (.51+\dx,-\dy) node[color=black] {\footnotesize $x_1$};
\draw (.62+\dx,-\dy) node[color=black] {\footnotesize $x_0$};
\draw (1.,-\dy) node[color=black] {\footnotesize $1$};

%
%
\def\gx{.05}
\def\gy{.04}
\draw (-\gx,0) node[color=black] {\footnotesize $0$};
\draw (-\gx,.35) node[color=black] {\footnotesize $y$};
\draw (-\gx,.25) node[color=black] {\footnotesize $y_0$};
\draw (-\gx,.44) node[color=black] {\footnotesize $y_1$};
\draw (-\gx,.58) node[color=black] {\footnotesize $y_2$};
\draw (-\gx,.70) node[color=black] {\footnotesize $y_3$};
\draw (-\gx,.81) node[color=black] {\footnotesize $y_4$};
\draw (-\gx,1.) node[color=black] {\footnotesize $1$};
\end{pgfonlayer}

\begin{pgfonlayer}{foreground layer}
\draw[thick,color=black] (0,0)--(0,1)--(1,1)--(1,0)--(0,0);
\end{pgfonlayer}

\end{tikzpicture}
\end{adjustbox}